\documentclass[letterpaper, 10pt, conference]{ieeeconf}

\IEEEoverridecommandlockouts                              
\usepackage{graphicx}

\usepackage{amsmath,amsthm}
\usepackage{cite}
\usepackage{amssymb}
\usepackage{newtxmath}
\usepackage{subfigure}
\usepackage{subfigmat}
\usepackage{epstopdf}
\usepackage{xcolor}
\usepackage{algorithm}
\usepackage{algorithmic}
\usepackage[utf8]{inputenc}
\allowdisplaybreaks

\usepackage{multirow,multicol,threeparttable, pgfplots}
\usepackage{booktabs}
\usepackage{enumerate}
\usepackage{color}

\def\BibTeX{{\rm B\kern-.05em{\sc i\kern-.025em b}\kern-.08em T\kern-.1667em\lower.7ex\hbox{E}\kern-.125emX}}

\newtheorem{proposition}{Proposition}[section]

\renewcommand{\theenumi}{\textbf{Case} \arabic{enumi}}

\title{\LARGE \bf Multi-UAV Routing for Persistent Intelligence Surveillance \& Reconnaissance Missions}

\author{Satyanarayana G. Manyam,  Steven Rasmussen,  David W. Casbeer, \\ Krishnamoorthy Kalyanam and Suresh Manickam%
\thanks{S. G. Manyam is a NRC Fellow, Air Force Research Laboratory, Wright-Patterson AFB, OH, USA {\tt\small msngupta@gmail.com}}%
\thanks{S. Rasmussen is a Principal Engineer for Miami Valley Aerospace LLC, Air Force Research Laboratory, Wright-Patterson AFB, OH, USA
        }%
\thanks{K. Kalyanam is a Research Scientist with the 
InfoSciTex Corporation, a DCS company, Dayton, OH, USA
        }%
\thanks{D. Casbeer is a Research Engineer with the United States Air Force, Air Force Research Laboratory, Wright-Patterson AFB, OH, USA
        }%
\thanks{S. Manickam is a Scientist with the Aeronautical Development Establishment, DRDO, Bangalore, India. }%
\thanks{This work is sponsored through the US-India joint Defense project US-IN-AF-15-001}%
}

\begin{document}
\maketitle
\thispagestyle{empty}
\pagestyle{empty}

\begin{abstract}
We consider a Persistent Intelligence, Surveillance and Reconnaissance (PISR) routing problem, which includes collecting data from a set of specified task locations and delivering that data to a control station. Each task is assigned a refresh rate based on its priority, where higher priority tasks require higher refresh rates.  The UAV team's objective is to minimize the maximum of the delivery times of all the tasks' data to the control station, while simultaneously, satisfying each task's revisit period constraint. The centralized path planning problem for this PISR routing problem is formulated using mixed integer linear programming and solved using a branch-and-cut algorithm. Heuristics are presented to find sub-optimal feasible solutions that require much less computation time. The algorithms are tested on several instances and their performance is compared with respect to the optimal cost and computation time.
\end{abstract}

\section{Introduction} \label{sec:intro}

Unmanned Aerial Vehicles (UAVs) are a natural choice for deployment in many military Intelligence, Surveillance and Reconnaissance (ISR) missions\cite{schanzreaper}. A typical ISR scenario involves monitoring a set of task locations for an indefinitely long period of time. These task locations can be buildings, road networks bordering a military base etc. Since these task locations are spatially dispersed, UAVs can be deployed to visit them regularly and ferry the information such as images, videos, sensor data etc. to a control station. This data needs to be delivered to the control station at regular intervals. The importance level of each of the task locations may vary from minimal to highly critical. While scheduling these monitoring missions, it is imperative to schedule the UAV to visit important task more frequently than the ones with lesser significance. 

We consider a persistent monitoring scenario, where a set of task locations needs to be visited persistently by multiple UAVs. We assume all the available UAVs are homogeneous. Therefore, there is no difference between visits by different UAVs to the same task. We are interested in two metrics \textit{viz.} data latency or delivery time (to the control station) and revisit rate or revisit period. We define the data delivery time (or latency time) as the time elapsed from collection of data from a task to the time the data is delivered to the control station. The data needs to be delivered at control stations as early as possible. The revisit period is the time between two successive visits to a task location; a task with higher priority needs to have a smaller revisit period compared to a low priority.

\textit{Prior work:} Several variants of persistent routing problem were addressed in \cite{chevaleyrepatrol, nigampers, elmaliachareapatrol, stumpmulti, pasqualetticooperative, smitht2012pers, smith2014persistent, aksaraydistributed, ghosepers}. Strategies to perform patrolling tasks by multiple agents on a network defined as a graph are presented in \cite{chevaleyrepatrol}. In \cite{nigampers, elmaliachareapatrol }, persistent surveillance of an area decomposed into cells is considered; \cite{nigampers} attempts to minimize the maximum time since last visit of all the cells, whereas \cite{elmaliachareapatrol} balances the frequency of visit to each cell. Persistent surveillance problems with tasks spatially distributed is posed as a vehicle routing problem with time windows in \cite{stumpmulti}. Patrolling strategies to minimize the refresh time of the viewpoints is presented in \cite{pasqualetticooperative}. Approximation algorithms are presented in \cite{smith2014persistent}, that minimize the maximum weighted latency (time between successive visits) in discrete environments. In \cite{aksaraydistributed}, the authors attempts to minimize time between two consecutive visits to partitioned regions while satisfying temporal logic constraints of each agent. A persistent routing scenario where some regions need more visitation than others is addressed in \cite{ghosepers}, and a policy to achieve that is proposed. 

In this article, we consider a persistent routing of tasks that are spatially distributed. Also the data collected at the task locations needs to be delivered at a control station (also referred to as depot). In the existing literature concerning persistent routing, the concept of a control station is not considered and delivery time is not addressed. 

We model this persistent routing problem as a multiple traveling salesman problem with revisit period constraints, and formulated as a mixed integer linear programming (MILP) problem. The contributions of this article are the following: ($i$) We present a novel formulation addressing two important metrics, delivery time and revisit period for the tasks in ISR missions and model it as a multiple vehicle path planning problem with cycle length constraints. ($ii$) We present two different MILP models to find optimal solutions to the corresponding path planning problem. The two MILP models constitute novel constraints to address the cycle length limits, which could be applied to any general routing problem involving constraints on cycle length. ($iii$) A heuristic via assignment-tree search is presented that produces good sub-optimal solutions, and it could be easily generalized to address different cost functions and/or constraints. ($iv$) We test the algorithms on several random instances and computational results are presented. 

The PISR routing problem is closely related to the distance constrained and fuel constrained vehicle routing problems \cite{karavrp, sundarfuelacc16}. This problem differs from these as follows: rather than the constraints being dictated by the UAV, the constraints on a tour are dictated by the tasks a UAV visits, which is a harder constraint to deal with. Also the cost function chosen here to capture the latency requirements is different from the cost of distance constrained VRPs.

\section{Preliminaries and Assumptions}
Here, we aim to optimize the total cost of the persistent ISR mission by centralized planning. We are looking at solving the problem before the mission begins and assigning the tasks for each UAV/agent to perform in a pre-specified sequence. Ideally, in a persistent routing scenario, the objective is to optimize the chosen metrics over an infinite time horizon. To plan the mission and schedule tasks to be serviced by each UAV, we need to generate an infinite sequence of visits for each UAV, which is infeasible. To overcome this, one may periodically solve a receding horizon problem, generate the task sequences and schedule the UAVs as proposed in \cite{stumpmulti}. However, reliable UAV-to-UAV communication links would be required, as well as the precise location of all the UAVs each time the planner is executed. 

A critical assumption we make for \textit{a-priori} mission planning is the following: we restrict each task to be serviced by the same UAV throughout the mission. Each UAV performs the tasks assigned and returns to the control station, and repeats exactly the same sequence throughout the mission. For example consider a mission with two UAVs and five tasks $\{t_1, t_2, t_3, t_4, t_5\}$, and let $t_d$ represent the control station (or depot). A sample assignment for two vehicles $V_1$ and $V_2$ is as follows: $V_1: t_d \rightarrow t_1 \rightarrow t_2$ and $V_2: t_d \rightarrow t_3 \rightarrow t_4 \rightarrow t_5$. Here, if the time of travel for the sequence $t_d \rightarrow t_1 \rightarrow t_2 \rightarrow t_d$ is $R_1$ seconds, then tasks $t_1$ and $t_2$ are serviced once every $R_1$ seconds.
\begin{figure}[htpb]
\centering
\includegraphics[width=0.8\columnwidth]{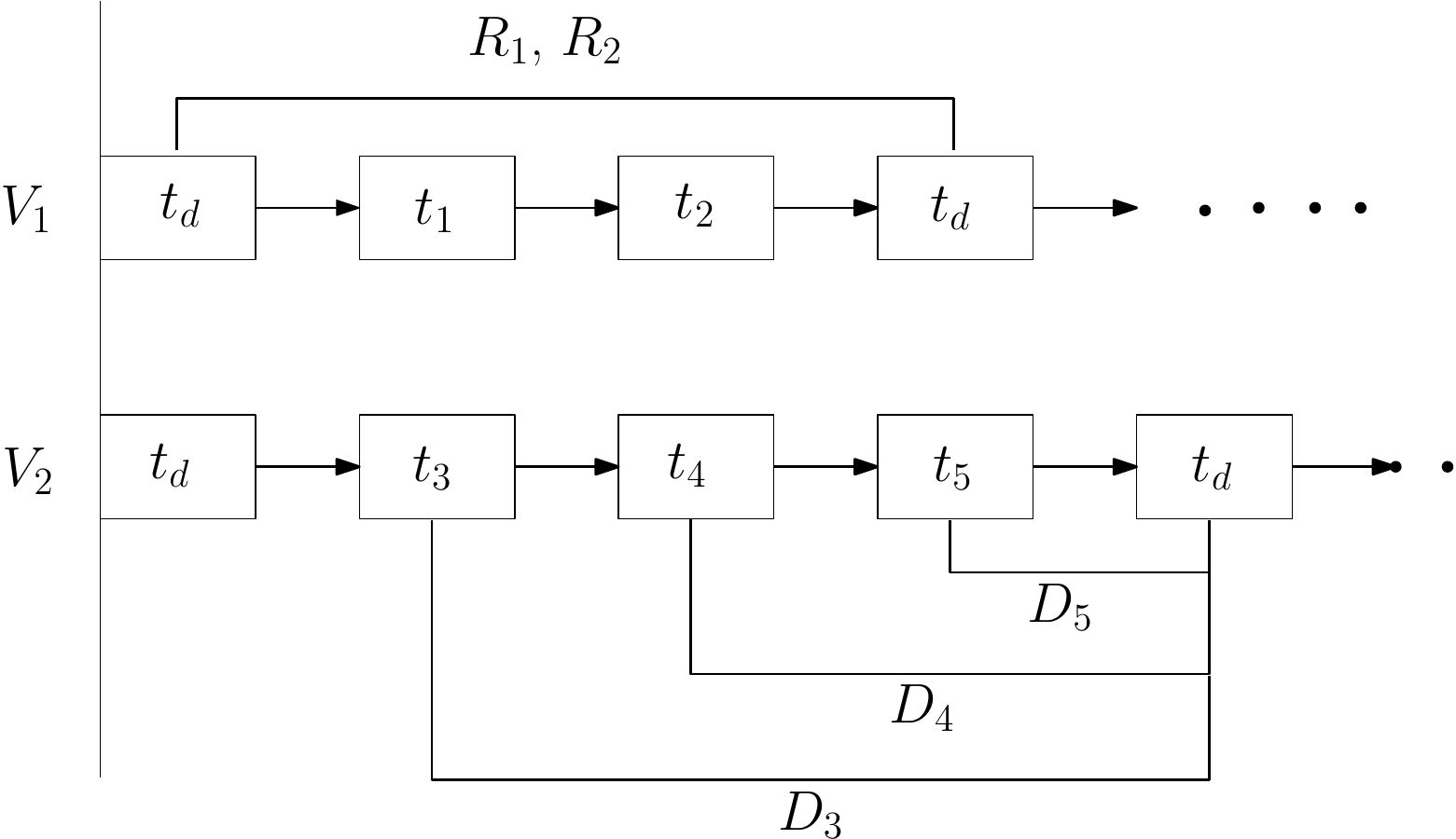}
\caption{Tasks schedule for two vehicle scenario}
\label{fig:tasks}
\end{figure}

Some of the advantages of this class of solutions are the following: we do not need to have communication between UAVs to update the scheduled tasks at each planning time interval. Whenever a UAV breaks down or needs refueling, the operator precisely knows which tasks are not being serviced, such that a contingency plan could be scheduled. The revisit period for each task is exactly known based on the tasks and the sequence assigned to the UAVs. Also under this restriction, the data that is collected from a task is delivered to the control station before its next visit. This is not guaranteed in the unrestricted case, where a task could be serviced by two different UAVs at successive visits. Along with the above advantages, there is a shortcoming; the cost one would optimize with this restriction could be different from the cost of the unrestricted case. However, due to its advantages in planning the mission and implementing, we pursue the restricted case where each tasks is assigned to one of the UAVs, and it is serviced by the same UAV throughout the mission.

There are two important metrics that needs to be addressed in PISR missions. The first one is the data delivery time or data latency ($D_i$) for each task $t_i$; $D_i$ is the elapsed time from when a task is completed until the vehicle returns to the depot. This is not the direct travel time between the task and the depot, as the vehicle may service other tasks before returning to depot. This is illustrated in Fig. \ref{fig:tasks}, the delivery time for the tasks $t_3$, $t_4$ and $t_5$ are shown as $D_3$, $D_4$ and $D_5$ respectively. 

The other metric that we consider is the revisit period of each task. It is the time between two successive visits of a task by an UAV. Since each UAV visits the same set of tasks and repeats, the revisit period is the same for every successive visit. The revisit periods $R_1$ and $R_2$ are illustrated in Fig. \ref{fig:tasks} for the tasks $t_1$ and $t_2$. Based on the importance or risk levels, some of the tasks require higher revisit rates than others. We aim to solve the PISR routing problem where a maximum limit on revisit period, $R_i$,  is specified for each task $t_i$. We want the data delivered at the control station to be as fresh as possible, which requires the delivery times to be as small as possible. To accomplish this, we minimize the maximum delivery time of all the tasks: $\min \max_{i \in T} D_i$. This cost is different from minimizing total cost of the paths, and it is a measure of the delivery time of the first task that is serviced.

\section{Problem Formulation} \label{sec:prbfrm}
In this section, we define the path planning problem for the PISR missions in more detail. Also we present a mixed integer linear programming (MILP) formulation to find the optimal solution to the path planning problem. We model the MILP using node based and arc based formulations; these models are akin to the models in \cite{fischettitsptw} and \cite{karavrp} used to solve the traveling salesman problem with time windows and the distance constrained vehicle routing problem. Similar formulations were also used to solve fuel constrained multiple vehicle routing problem in \cite{sundarfuelacc16}. In these articles, the constraints on the length of a tour starting from a depot are constrained. In the formulation presented here, the length of a tour starting from a depot to each task and the length starting from the task returning to the depot together are constrained. The novelty of this formulation lies in modeling and constraining these two lengths together. Also this could be applied to other routing problem which requires to handle cycle lengths such as min-max traveling salesman problem.

Let $T = \{t_1, t_2, \ldots t_n \}$ represent a set of tasks, and as in the previous example $d$ be the index referring to the depot. We define the problem on a graph $G(V, E)$. $V$ is the set of nodes $V =T \cup \{d\}$, and $E$ is the set of edges between every pair of nodes in $V$. Let $n_v$ represent the number of UAVs available for the mission. The problem can be stated as the following: find at most $n_v$ cycles that minimizes the maximum delivery time such that, $(i)$ each task $T$ is covered by one cycle, and $(ii)$ if a task $t_i$ assigned to one of the UAVs, $v$, with cycle length $L_v$, then $L_v \le R_i, \, \forall i \in T$. 

\subsection{MILP Formulation}
In the MILP formulation, we use a set of binary variables $x_{ij}$'s and two sets of real variables $u_i$'s and $v_i$'s. Each variable $x_{ij}$ corresponds to an edge $(i,j)$, and $x_{ij}=1$ if edge $(i,j)$ is in any of the UAV cycles.  Otherwise, $x_{ij}=0$. For a particular cycle (assignment), the variables $u_i$'s denote the travel time from the depot to the task $t_i$, and $v_i$'s represent the return travel time along the cycle from task $t_i$ to the depot; this is illustrated in Fig. \ref{fig:uv_vars}. Let $c_{ij}$ represent the time elapsed from task $t_i$ to task $t_j$. Here, $c_{ij}$ includes the time of travel between the tasks and the time to perform task $t_j$. Now we present the MILP formulation using degree constraints, sub-tour elimination constraints (SEC) and revisit period constraints. \\

\noindent \emph{Degree constraints:} 
\begin{align}
& \sum_{j \in V}x_{ij} =1 \mbox{ and } \sum_{j \in V}x_{ji} =1, \,\, \forall i \in T, \label{f1cons:deg1} \\
& \sum_{j \in T}x_{dj} \le n_v \mbox{ and } \sum_{j \in T}x_{jd} \le n_v, \label{f1cons:depdeg} \\ 
& x_{ij} \in \{0,1 \} \,\, \forall (i,j) \in E. \label{f1cons:xbin}
\end{align} 
Constraints in Equation (\ref{f1cons:deg1}) state that for every node representing a task, there should be one incoming edge and one outgoing edge. The constraints (\ref{f1cons:depdeg}) state that there should be a maximum of $n_v$ number of incoming and outgoing edges for the depot node; and the binary constraints on the $x_{ij}$ variables are in (\ref{f1cons:xbin}).

\begin{figure}[htpb]
\vspace{0.2cm}
\centering
\includegraphics[width=0.9\columnwidth]{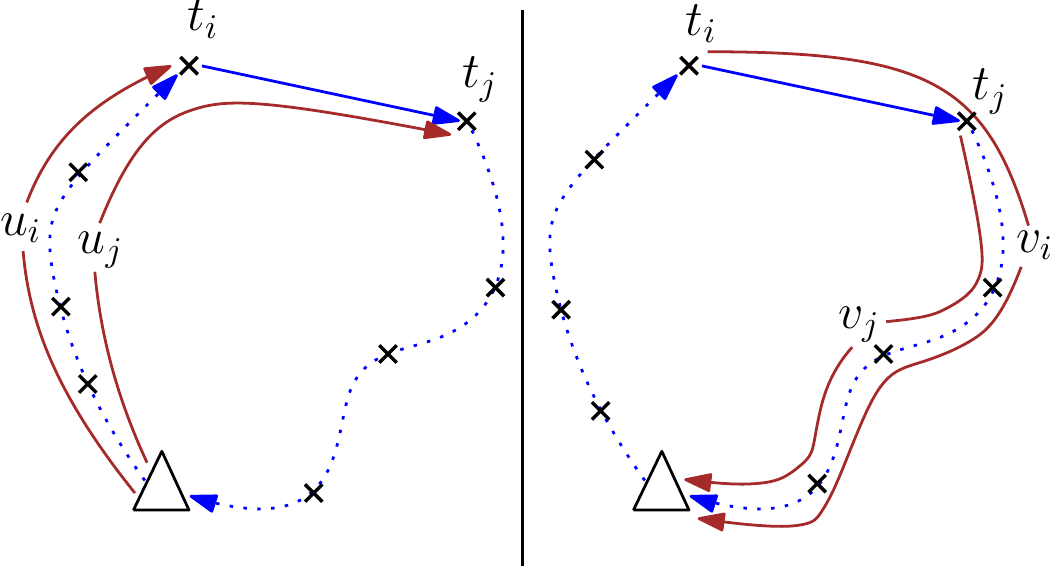}
\caption{Illustrating the variables $u_i$'s and $v_i$'s. Note that, the figure shows one cycle corresponding to one of the UAV paths, which is only a part of the multi UAV tour for the PISR routing problem.}
\label{fig:uv_vars}
\end{figure}

\noindent \emph{Sub-tour elimination constraints:}
\begin{align}
& u_i - u_j + c_{ij} \le M(1-x_{ij}), \,\, \forall i \in V,\, j \in T \label{f1cons:secu1} \\
& c_{di} \le u_i \le R_i-c_{id}, \,\, \forall i \in T, \label{f1cons:secu2} \\
& v_j - v_i + c_{ij} \le M(1-x_{ij}), \,\, \forall i \in T, \, j \in V \label{f1cons:secv1} \\
& c_{id} \le v_i \le R_i-c_{di}, \,\, \forall i \in T. \label{f1cons:secv2} 
\end{align} 
With just the degree constraints, the MILP may produce solutions containing sub-tours; a sub-tour is an assignment where a subset of tasks are connected as a cycle, but are isolated from the depot. One may refer to \cite{applegatebook} for further reading on sub-tours. One can remove these infeasible solutions using the sub-tour elimination constraints. To this end, we use the inequalities (\ref{f1cons:secu1}, \ref{f1cons:secu2}) similar to the \emph{MTZ}-constraints used to solve traveling salesman problem\cite{mtz,applegatebook}.
If an edge $(i,j)$ is present in the tour, then the value of the variable $u_j$ should be at least  the sum of travel time from depot to task $t_i$ and the travel time from task $t_i$ to $t_j$. If $x_{ij}=1$, then $u_j \ge u_i +c_{ij}$. This is enforced by the constraints in (\ref{f1cons:secu1}), where $M$ (referred to as big-$M$ in literature) is a constant of arbitrarily high value. When $x_{ij}$ is zero, the constraint (\ref{f1cons:secu1}) is trivially satisfied. For every $t_i$, the minimum value of $u_i$ is the direct travel time from the depot to the task, and the maximum is $R_i - c_{id}$. These lower and upper limits on $u_i$'s are imposed by (\ref{f1cons:secu2}). Inequalities \eqref{f1cons:secv1}, \eqref{f1cons:secv2} are equivalent to \eqref{f1cons:secu1}, \eqref{f1cons:secu2}, using the variables $v_i$'s instead. Inequality (\ref{f1cons:secv1}) states that, when $x_{ij}=1$, then $v_j$ should be less then $v_i - c_{ij}$. Either the set of constraints \eqref{f1cons:secv1}, \eqref{f1cons:secv2} or \eqref{f1cons:secu1}, \eqref{f1cons:secu2} are sufficient to eliminate the sub-tours, but we need both of these to formulate the revisit period constraint, which need both sets of variables, $u_i$'s and $v_i$'s. \\

\noindent \emph{Revisit period constraints:}
\begin{align}
& u_i + v_j \le R_i, \, \forall i \in T. \label{f1cons:revprd}
\end{align} 
For each task, $u_i$ is the time of travel from the depot to the task $t_i$, and $v_i$ is the return time of travel from the task to the depot. Hence, the sum of these two variables gives the time of travel of the full cycle which covers the task $t_i$. Therefore, inequalities (\ref{f1cons:revprd}) enforce the revisit period constraints for all of the tasks in $T$. \\

\noindent \emph{Objective:}
\begin{align}
&\quad\mbox{Minimize }  z \nonumber  \\
& v_{i} \le z, \,\, \forall i \in T. \label{f1cons:maxvi}
\end{align}
The variables $v_i$'s also are equal to the delivery time of the data collected from each task. To minimize the maximum of all the delivery time, we introduce an auxiliary variable $z$, which is needed to formulate the $\min$-$\max$ objective. The objectives $\min z$ and the inequality (\ref{f1cons:maxvi}) together minimize the maximum of all of the delivery times ($v_i$'s).

The MILP formulation for the PISR routing problem is the following:
\begin{flalign}
&(\mathcal{F}_1) \quad\mbox{Minimize }  z \nonumber  \\
&\mbox{subject to: (\ref{f1cons:deg1}) - (\ref{f1cons:maxvi})}  \nonumber 
\end{flalign}

In the above formulation, the big-$M$ in the constraints \eqref{f1cons:secu1}, \eqref{f1cons:secv1} is known to cause computational problems \cite{fischettitsptw, melkoteintegrated}, and hence make the MILP model computationally less efficient. We propose a second formulation without big-$M$ constraints and compare the computational performance of these two formulations.

\subsection{Formulation based on arcs ($\mathcal{F}_2$)}

Here, we use the binary variables $x_{ij}$'s similar to the previous formulation, and the real variables $y_{ij}$'s and $w_{ij}$'s  $\forall i,j \in V$ are used instead of $u_i$'s and $v_i$'s.   Variables $y_{ij}$ represent the travel time from depot to the task $t_j$ if the edge $(i,j)$ is selected in the assignment, \emph{i.e.,} $x_{ij} = 1$.  Also, when $x_{ij}=1$ the variable $w_{ij}$ is equal to the return travel time from $t_i$ to the depot. For each task $t_i$, only one of the variables in the set $y_{ij}, \, j \in T$ and one of the variables in the set  $w_{ij}, \, j \in T$'s are non-zero. The arc based MILP formulation to solve the PISR routing problem is the following:

\begin{flalign}
&(\mathcal{F}_2) \quad\mbox{Minimize }  z \nonumber  \\
&\mbox{subject to} \nonumber \\
& \sum_{j \in V}x_{ij} =1 \mbox{ and } \sum_{j \in V}x_{ji} =1, \,\, \forall i \in T, \label{cons:deg1} \\
& \sum_{j \in T}x_{dj} \le n_v \mbox{ and } \sum_{j \in T}x_{jd} \le n_v, \label{cons:depdeg} \\ 
& \sum_{j \in V}y_{ij} - \sum_{j \in V}y_{ji} = \sum_{j \in V}c_{ij}x_{ij}, \,\, \forall i \in T, \label{cons:secy1} \\
& y_{di} = c_{di}x_{di}, \,\, \forall i \in T, \label{cons:secy2} \\
& 0 \le y_{ij} \le R_j x_{ij}, \,\, \forall i \in V, \, j \in T, \label{cons:secy3} \\
& \sum_{j \in V}w_{ji} - \sum_{j \in V}w_{ij} = \sum_{j \in V}c_{ji}x_{ji}, \,\, \forall i \in T, \label{cons:secw1} \\
& w_{id} = c_{id}x_{id}, \,\, \forall i \in T, \label{cons:secw2} \\
& 0 \le w_{ij} \le D_i x_{ij}, \,\, \forall i \in T,\, j \in V \label{cons:secw3} \\
& \sum_{j \in V}y_{ji} + \sum_{j \in V} w_{ij} \le R_i, \,\, \forall i \in T, \label{cons:revprd} \\
& w_{ij} \le z, \,\, \forall i \in T, \, j \in V, \label{cons:maxdt} \\
& x_{ij} \in \{0,1 \} \,\, \forall (i,j) \in E. \label{cons:xbin}
\end{flalign}

\begin{figure}[htpb]
\vspace{0.2cm}
\centering
\includegraphics[width=0.9\columnwidth]{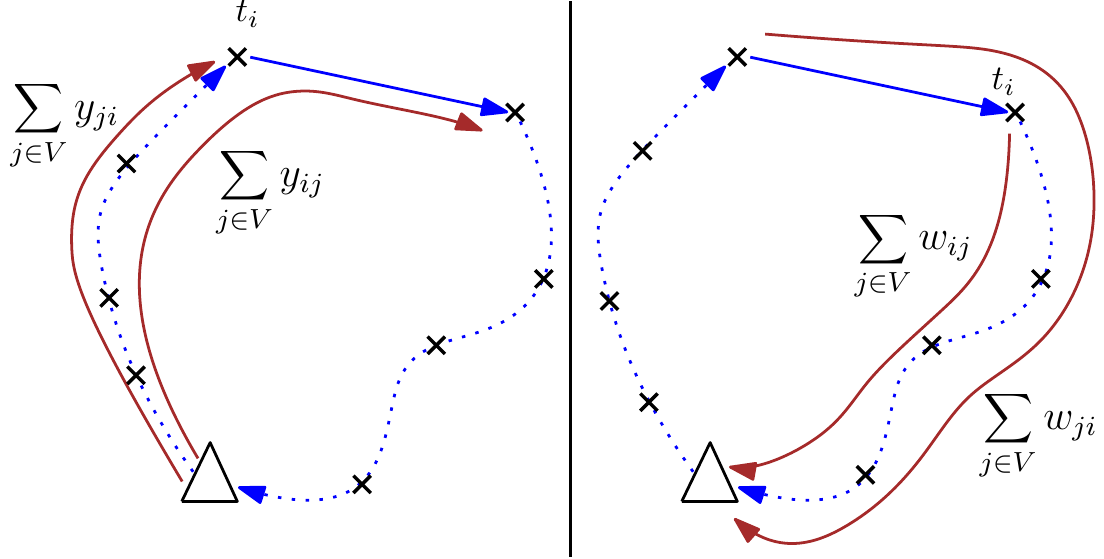}
\caption{Illustrating the variables $y_{ij}$'s and $w_{ij}$'s}
\label{fig:yij}
\end{figure}

Here, constraints in the Equation (\ref{cons:deg1}) are the degree constraints, which enforce that only one incoming and one outgoing edge exists for any task. Equation (\ref{cons:depdeg}) imposes the maximum number of edges going out of and coming into depot to be $n_v$. The time elapsed from leaving the depot to the end of the task $t_i$ is given by the summation $\sum_{j \in V}y_{ji}$, and the time from leaving the depot to the end of the task that is serviced after $t_i$ is given by the summation $\sum_{j \in V}y_{ij}$; these two summation are illustrated in Fig. \ref{fig:yij}. Constraints (\ref{cons:secy1}) and (\ref{cons:secy2}) ensure that the difference between these two should be equal to the time between $t_i$ and task serviced after $t_i$. These constraints are needed to eliminate the sub-tours. Equations (\ref{cons:secy3}) ensures the variables $y_{ij}$'s are non-negative always, and nonzero only when $x_{ij}=1$. Constraints in  (\ref{cons:secw1}) are the counterpart of (\ref{cons:secy1}), where the summations $\sum_{j \in V}w_{ij}$ and $\sum_{j \in V}w_{ji}$ are the return travel times to the depot from the task $t_i$  and the task serviced before $t_i$ respectively. The difference between these two should be equal to the time between these two tasks, and this is imposed by constraints in (\ref{cons:secw1}) and (\ref{cons:secw2}). The non-negative and non-zero constraints on variables $w_{ij}$'s are enforced by (\ref{cons:secw3}). The set of constraints \eqref{cons:secy1} - \eqref{cons:secy3} or \eqref{cons:secw1} - \eqref{cons:secw3} are sufficient to eliminates the sub-tours, however we need both of these to implement the revisit rate constraints.

The summation $\sum_{j \in V}y_{ji}$ is the time of travel for a UAV from the depot to the task $t_i$, and the summation $\sum_{j \in V} w_{ij}$ is the return travel time from the task $t_i$ to the depot. If UAV$_v$ is visiting task $t_i$, then these two summations adds up to the total time elapsed to service all the tasks assigned to UAV$_v$. Constraints (\ref{cons:revprd}) enforces the maximum limit on the revisit period for each of the tasks. The variables $w_{ij}$'s are the delivery time of the tasks, hence the objective function $\min z$ and the inequalities (\ref{cons:maxdt}) together minimizes the maximum delivery time.

\begin{proposition}
Inequalities \eqref{cons:secy2} - \eqref{cons:secy3} and \eqref{cons:secw2} - \eqref{cons:secw3} in formulation $\mathcal{F}_2$ can be strengthened using the following inequalities:
\begin{flalign}
& y_{ij} \le (R_j-c_{jd})x_{ij}, \,\, \forall i,j \in V, \label{cons:streny1} \\
& y_{id} \le R_ix_{id}, \,\, \forall i \in V, \label{cons:streny2} \\
& y_{ij} \ge (c_{di} + c_{ij})x_{ij}, \,\, \forall i,j \in V,, \label{cons:streny3} \\
& w_{ij} \le (R_i-c_{di})x_{ij}, \,\, \forall i,j \in V,, \label{cons:strenw1} \\
& w_{di} \le R_ix_{di}, \,\, \forall i \in V, \label{cons:strenw2} \\
& w_{ij} \ge (c_{ij} + c_{jd})x_{ij}, \,\, \forall i,j \in V. \label{cons:strenw3}
\end{flalign}
\end{proposition}
\begin{proof}
When $j$ is not the depot index, sum of the time from leaving the depot to the end of task $t_j$ and $c_{jd}$, direct travel time from the end of task $t_j$ to the depot should be less than the revisit period limit for the task $t_j$. This inequality is expressed in (\ref{cons:streny1}), and since $c_{jd}$ is non-negative, (\ref{cons:streny1}) is tighter than (\ref{cons:secy3}). Inequality (\ref{cons:streny2}) is same as (\ref{cons:secy3}) where $j$ is the depot index. Inequality (\ref{cons:streny3}) indicates that when $i$ and $j$ are not depot indices and $x_{ij}=1$, then the time to travel from depot to $t_j$ should be at least  the sum of time to travel from depot to $t_i$ and time of travel from $t_i$ to $t_j$.  

Inequalities (\ref{cons:strenw1}) - (\ref{cons:strenw3}) are counterpart of the inequalities (\ref{cons:streny1}) - (\ref{cons:streny3}) for the $w_{ij}$ variables. Inequality (\ref{cons:strenw1}) states that, when $x_{ij}=1$, sum of the direct time of travel from depot to $t_i$ and time from $t_i$ to depot along the cycle should be less than the revisit period limit of $t_i$. When $x_{ij}=1$, $w_{ij}$ is the time from end of task $t_i$ to returning to the depot, which should be at least the sum of time to travel from $t_i$ to $t_j$ and time to travel from $t_j$ to the depot; this is enforced by (\ref{cons:strenw3}).
\end{proof}

\begin{proposition}
To minimize the maximum delivery time, inequalities (\ref{cons:maxdt}) can be replaced with the following inequalities:
\begin{flalign}
w_{di} - c_{di} \le z, \,\, \forall i \in T. \label{cons:strenmaxdt}
\end{flalign}
\end{proposition}
\begin{proof}
Clearly, the first task serviced has the highest delivery time. If $t_i$ is the first task, then $w_{di} - c_{di} $ is the delivery time for $t_i$. Therefore, the objective $\min z$ along with inequalities (\ref{cons:strenmaxdt}) minimizes the maximum delivery time. Note that, if $t_i$ is not the first task visited, then $w_{di}=0,$ and (\ref{cons:strenmaxdt}) is trivially satisfied.
\end{proof}

We present the strengthened arc based formulation as follows:
\begin{flalign}
&(\mathcal{F}_3) \quad\mbox{Minimize }  z \nonumber  \\
&\mbox{subject to} \nonumber \\
& (\ref{cons:deg1}) - (\ref{cons:secy2}), (\ref{cons:secw1}),(\ref{cons:secw2}), (\ref{cons:revprd}), \mbox{and} (\ref{cons:xbin}) - \ref{cons:strenmaxdt}).
\end{flalign}

\section{Assignment Tree Search Heuristic}
In this section, we present a heuristic to solve the PISR routing problem. The heuristic is a greedy assignment tree search, based on the prior work in \cite{karamanplanning,rasmussentree}, for planning missions involving multiple UAVs. Here, we present a synopsis of the tree search algorithm, however one can refer to \cite{karamanplanning,rasmussentree} for further details. This tree search follows a best first search pattern until it finds a feasible assignment. At the root node, the algorithm creates branches and a child node at each branch. Each child corresponds to an assignment of one of the tasks to one of the available UAVs. The number of child nodes are all possible ways to select one unassigned task and assign it to a UAV. Among all the child nodes, the algorithm selects the node with the least cost, and repeats the branching process similar to the branching at root node. This process repeats until all the tasks are assigned.  At this point in the search the algorithm arrives at a "leaf" node, which corresponds to a feasible solution to the planning problem. Once, a feasible assignment is found the algorithm stores the solution as an incumbent solution. It proceeds to search the tree by evaluating the unexplored child nodes, and tries to find solutions of lower cost than the incumbent  solution. Also, the tree search prunes the branches before reaching a leaf node, if the current cost at the branch is more than the incumbent solution. The search is terminated either when it finds a feasible solution, or if it reaches a pre-specified maximum number of nodes to be explored. 

We use this tree search heuristic to find feasible paths for the PISR routing problem. The algorithm is adapted to find feasible paths, such that the cycle lengths of each UAV adheres to the revisit period constraints of the tasks the UAV is assigned, and minimizes the maximum delivery time of all the tasks.\\
\emph{Tree search heuristic:}
\begin{enumerate}[1.]
\item Initialize the problem at a root node with the following: the locations of the tasks, time of travel between the tasks, the maximum limit on the revisit period for each task, and the number of UAVs available.
\item \label{step2} Create child nodes ($n_i$), each corresponds to an assignment of a task to an available UAV. Each node corresponds to a list of assignments for each UAV. 
\begin{enumerate}[(i)]
	\item Compute the current cost of each child node $C(n_i)$, which is the maximum of the delivery times of all the tasks. For example, a UAV is assigned tasks in the order $t_{s_1}, t_{s_2}, t_{s_3}$, the maximum of the delivery times is the sum $c_{{s_1}{s_2}} + c_{{s_2}{s_3}} + c_{{s_3}{d}}$. ( \textit{Here, $c_{ij}$ is the sum of the time of travel between $t_i$ and $t_j$ and the time to performs task $t_j$.})
	\item Check if the assignments violates the revisit period constraints of all the tasks assigned so far. For example if the tasks, $t_{s_1}, t_{s_2}, t_{s_3}$ are assigned to an UAV, compute the travel time of the cycle $c_{d{s_1}}+ c_{{s_1}{s_2}} + c_{{s_2}{s_3}} + c_{{s_3}{d}}$ is less than $R_{s_1}$, $R_{s_2}$ and $R_{s_3}$. If any UAV violates the revisit period constraints, assign an infinite cost to the child node.
\end{enumerate}
\item  \label{step3} To select a child node for further branching, we scale the cost based on two factors based on the current task and current UAV that are assigned at each node. The first scale $C_{s_1}$ is to force a task with the lowest revisit period limit to be assigned earliest to an UAV. $C_{s_1} = \frac{R_i}{R_{max}}i$, where $R_j$ is the maximum revisit period limit of the task $t_j$ that is assigned at the current node, and $R_{max}$ is the maximum revisit period limits of all the tasks. The second scale $S_{c_2}$ is to prioritize a UAV which is assigned a task with revisit limits in earlier assignments (at parent or above nodes); $S_{c_2} = 10^{-n_t}$, where $n_t$ is the number of revisit period constrained tasks assigned to the current UAV.
\item Select the child node with the lowest scaled cost $S_{c_1}S_{c_2}C(n_i)$, and repeat the branching, steps \ref{step2} - \ref{step3}, until a leaf node is found. Update the incumbent solution, and proceed to explore the unevaluated child nodes at the parent nodes and further until another leaf node is found. 
\item Exit the tree search when there are no child nodes to be evaluated, or the number of nodes evaluated reaches the specified limit, and output the solution with the lowest cost. 
\end{enumerate}

\section{Computational Results} \label{sec:compres}
The MILP formulations $\mathcal{F}_1$ and $\mathcal{F}_3$ are solved using branch and cut algorithm. The algorithms are implemented using CPLEX (version 12.6) with C++ API. CPLEX solves the MILP using branch and cut algorithm, which generates the feasible solutions (upper bounds) and lower bounds based on a solution to the dual problem iteratively, and outputs the optimal solution when the gap between the lower bound and upper bound converges to zero. All the simulations were run on a Macbook with Intel i5, $2.7$ GHz processor and 8 GB memory. We generated random instances by choosing task locations ($xy$-coordinates) from an uniform distribution in a square grid of size $4000 \times 4000$ meters. We have tested the algorithms on $30$ random instances,  $10$ each with $10$, $20$ and $30$ tasks and $4$ UAVs. We impose the revisit period constraints on $3$, $4$ and $5$ of the tasks for the instances with $10$, $20$ and $30$ tasks respectively. For all the instances, we have chosen the task farthest from the depot, and the revisit period limit is set to $1.1$ times the sum of the time from the depot to the farthest task and the task to the depot. We have selected the nearest $2$, $3$ and $4$ tasks for the instances with $10$, $20$ and $30$ tasks respectively to set the revisit period limits. We set the revisit period limit to be $1.1$ times the optimal cost of the traveling salesman problem solved on these tasks including the depot, with travel times as the cost of travel between tasks. We assume the UAVs travel at unit speed (one meter per second), and the Euclidean distance between the task locations is chosen to be the travel times between them. The tree search heuristic is implemented in C++, and the maximum number of nodes to be evaluated is set to one million. 

The computational results for the $30$ instances are shown in the Table \ref{tab:1}. The first and second column refers to the instance number and the number of tasks in the instance respectively. The third and fourth column refers to the cost of the solution and the computation time required by the formulation $\mathcal{F}_1$. The fifth and sixth column refers to the solution cost and the computation time required for solving using formulation $\mathcal{F}_3$. A time limit of one hour and $2.5$ hours are set for solving the instances with $20$ and $30$ tasks. For the instances where the algorithm could not find optimal solutions in the set time limit, the cost of the best found solution are listed. All the computation times reported are in seconds. The seventh and eighth columns refer to the cost of the first feasible solution found (referred to as best first search solution) and the corresponding computation time by the tree search heuristic.  The ninth and tenth columns refer to the final cost of the heuristic solution and the computation time required after exploring one million nodes of the tree. 

With the formulation $\mathcal{F}_1$, the branch-and-cut algorithm could not converge with in the time limit for instances with $20$ and $30$ tasks. We could find optimal solutions for all the instances with $10$ and $20$ tasks using the formulation $\mathcal{F}_3$, and for $4$ out of $10$ instances with the $30$ tasks. Though tight feasible solutions are found, the lower bounds given by the LP relaxations of these formulations are not tight enough, and therefore the algorithm needs more computation time to converge. Finding better valid inequalities may solve this problem, which can be a future direction of this research. From the computational results, clearly the formulation $\mathcal{F}_3$ outperformed $\mathcal{F}_1$. The tree search heuristic generated a best first search solution within $11$ milliseconds and final solutions with in $11$ seconds for all the instances. Also the cost of these solutions are within $50\%$ of the optimal for most of the instances. This heuristic is well suitable for quick planning and onboard re-planning of the missions. Plots of the solutions found by solving the MILP formulation $\mathcal{F}_3$ and the heuristic for an instance with $20$ tasks are shown in Fig. \ref{fig:solutions}. One can see that the tasks $t_6, \, t_8$ and $t_{14}$ with tight revisit period limits lie on a UAV tour with the smallest tour length. The task $t_{10}$ also has revisit period limits, however the corresponding UAV also visits other tasks without violating the revisit period constraints of $t_{10}$.

\begin{figure}[htpb]
\centering
\subfigure[Optimal PISR solution solved using MILP formulation $\mathcal{F}_3$ ]{\includegraphics[width=0.70\columnwidth]{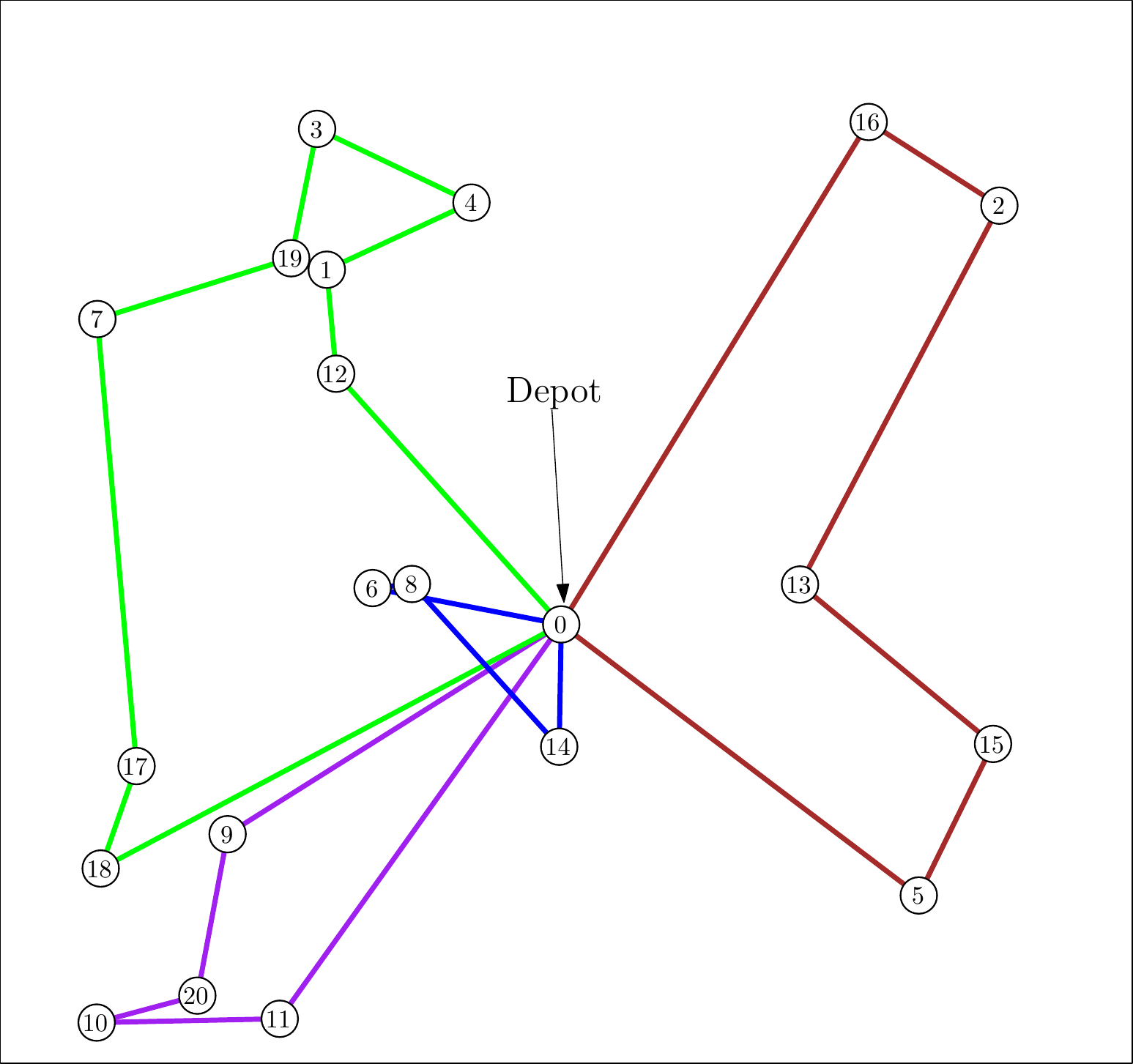}}
\subfigure[Solution produced by the assignment tree search heuristic]{\includegraphics[width=0.80\columnwidth]{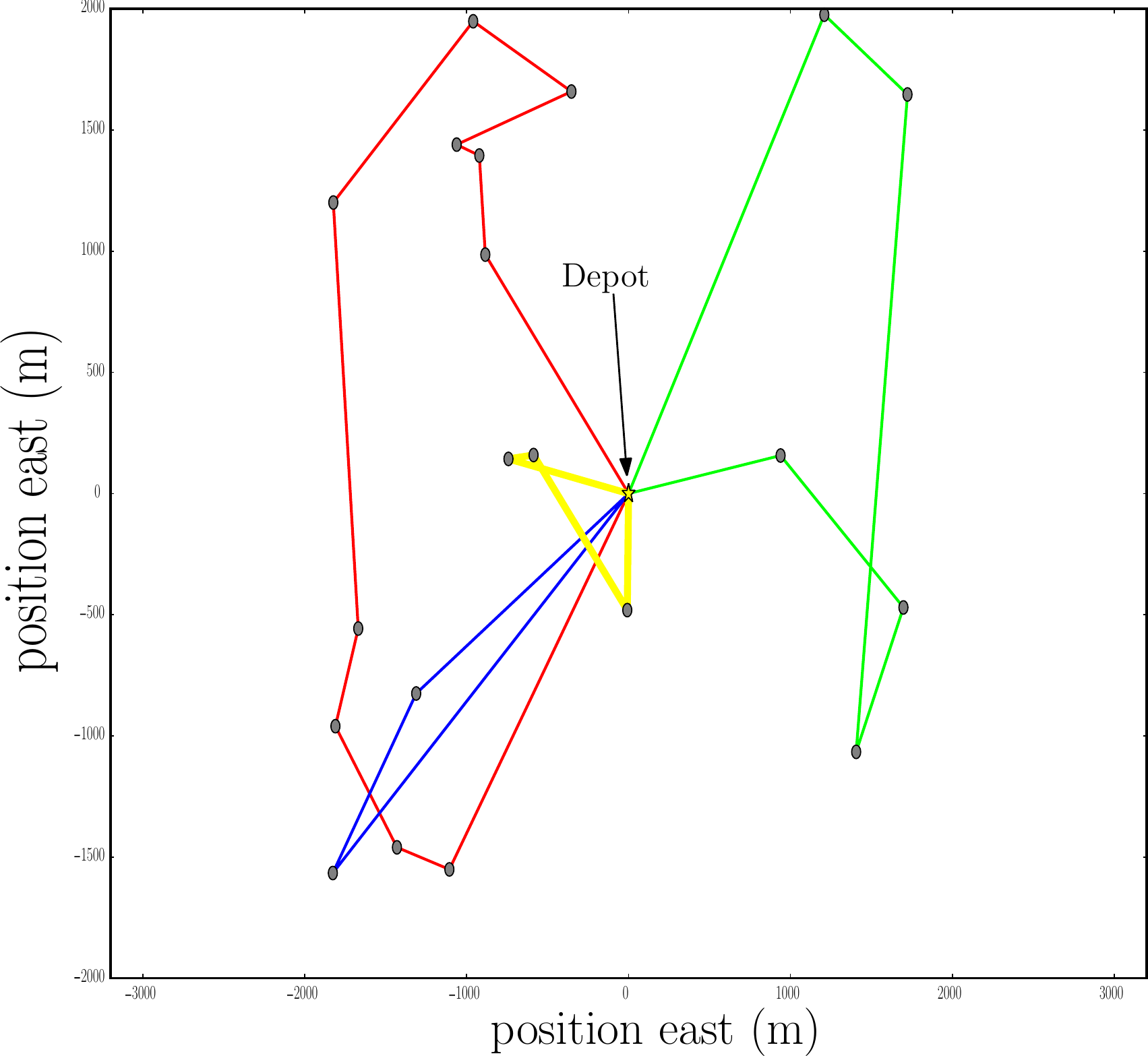}}
\caption{Solutions for an instance with $20$ tasks and $4$ UAVs}
\label{fig:solutions}
\end{figure}

\begin{table*}[htpb]
\vspace{0.2cm}
\centering
\caption{Computational Results using the MILP formulations and the Tree Search Heuristic}
\begin{tabular}{ccrrrrrcrc}
\toprule
Instance \# & $ |T|$ & \multicolumn{2}{|c|}{ $\mathcal{F}_1$} & \multicolumn{2}{|c|}{ $\mathcal{F}_3$}& \multicolumn{4}{|c}{ Tree Search Heuristic} \\
\cmidrule(l){3-10} \\
 & & Cost & CPU time & Cost& CPU time& BFS Cost& BFS CPU time& Final Cost& Final CPU time\\
\midrule
1 & 10 & 2786 & 0.07 & 2786  & 0.31 & 3773 & 0.001 & 2801 & 5.185 \\
2 & 10 & 4598 & 0.06 & 4598 & 0.22 & 5651 & 0.001 & 4598 & 5.255\\
3 & 10 & 4756 & 0.11 & 4756  &  0.99 & 5687 & 0.001 & 4756 & 5.444\\
4 & 10 & 5315 & 0.15 & 5315 &  1.01 & 7705 & 0.001 & 5315 & 5.379\\
5 & 10 & 3751 & 0.09 & 3751  &  0.89 & 4786 & 0.001 & 3751 & 5.295\\
6 & 10 & 3921 & 0.08 & 3921  &  0.29 & 4380 & 0.004 & 3921 & 5.200\\
7 & 10 & 3364  & 0.07 & 3364  &  0.16 & 5219 & 0.001 & 3364 & 5.263\\
8 & 10 & 3972 & 0.06 & 3972  &  0.53 & 7891 & 0.002 & 4515 & 5.275\\
9 & 10 & 3383 & 0.08 & 3383  &  0.18 & 5536 & 0.001 & 3383 & 5.475\\
10 & 10 & 3700 & 0.09 & 3700  &  0.59 & 4825 & 0.002 & 3700 & 5.275\\
11 & 20 & 6860$^*$ & 3600$^*$ & 6500 & 29 & 8043 & 0.005 & 7240 & 8.356\\
12 & 20 & 5260$^*$ & 3600$^*$ & 5260 & 153 & 5857 & 0.004 & 5592 & 8.089\\
13 & 20 & 6700$^*$ & 3600$^*$ & 6510 & 180 & 8993 & 0.004 & 8993 & 8.685\\
14 & 20 & 6740$^*$ & 3600$^*$ & 6680 & 14 & 9823 & 0.006 & 9252 & 8.585\\
15 & 20 & 6540$^*$ & 3600$^*$ & 6540 & 56 & 11704 & 0.005 & 8171 & 8.514\\
16 & 20 & 5960$^*$ & 3600$^*$ &  5960 & 18 & 9518 & 0.004 & 7866 & 8.685\\
17 & 20 & 6040$^*$  & 3600$^*$ & 5930 &  20 & 9220 & 0.005 & 8007 & 8.860\\
18 & 20 & 5800$^*$ & 3600$^*$ & 5800 & 238 & 9832 & 0.004 & 7148 & 8.390\\
19 & 20 & 6080$^*$ & 3600$^*$ & 6080 & 270 & 7992 & 0.005 & 7409 & 8.178\\
20 & 20 & 4100$^*$ & 3600$^*$ & 4100 & 219  & 7056 & 0.005 & 7056 & 9.031\\
21 & 30 & $^{**}$ & $^{**}$ & 7280 & 8623 & 26277 & 0.010 & 23418 & 10.328\\
22 & 30 & $^{**}$ & $^{**}$ & 7290 & 4967 & 12920 & 0.010 & 10682 & 10.585\\
23 & 30 & $^{**}$ & $^{**}$ & 6900$^*$& 10000$^*$ & 22461 & 0.012 & 16536 & 10.356\\
24 & 30 & $^{**}$ & $^{**}$ & 7160 & 8070 & 21070 & 0.010 & 17186 & 10.266\\
25 & 30 & $^{**}$ & $^{**}$ & 6430$^*$ & 10000$^*$ & 9394 & 0.011 & 8911 & 10.545\\
26 & 30 & $^{**}$ & $^{**}$ & 7420$^*$ & 10000$^*$ & 16488 & 0.010 & 15315 & 10.383\\
27 & 30 & $^{**}$ & $^{**}$ & 7190$^*$ & 10000$^*$ & 11785 & 0.010 & 9663 & 10.681\\
28 & 30 & $^{**}$ & $^{**}$ & 7490$^*$ & 10000$^*$ & 11613 & 0.011 & 11613 & 10.530\\
29 & 30 & $^{**}$ & $^{**}$ & 7160 & 4886$^*$ & 10159 & 0.010 & 10159 & 10.542\\
30 & 30 & $^{**}$ & $^{**}$ & 6560$^*$ & 10000$^*$ & 10413 & 0.011 & 10031 & 10.954\\
\bottomrule\\
\multicolumn{10}{l}{ $^*$\textit{Instances are not converged to the optimal solution with in the specified time limit; the best found solution is reported.}} \\
\multicolumn{10}{l}{ $^{**}$\textit{Instances could not find a solution within the set time limit.}} 
\end{tabular}
\label{tab:1}
\end{table*}

\section{Conclusion} \label{sec:conc}
We considered a path planning problem for PISR missions that involves multiple UAVs collecting data from spatially dispersed tasks, and delivering at a depot. We have modeled this as an optimization problem to minimize the maximum delivery time for all the tasks while satisfying the revisit period constraints for the high priority tasks. To find optimal solutions, we presented two MILP formulations $\mathcal{F}_1$ and $\mathcal{F}_3$, which include novel constraints to satisfy revisit period limits. These formulations are solved using branch and cut algorithm, and it could find optimal solutions for instances up to thirty tasks. Also, we presented a heuristic based on assignment tree search; it produces sub-optimal solutions which require only a few seconds of computation time. The heuristic could find feasible solutions for all the instances with in $10$ milliseconds. For the missions where onboard re-planning is necessary due to change in the tasks or locations, this heuristic is well suitable for quick onboard re-planning. The future directions of this research include finding better valid inequalities for the formulations to strengthen the lower bounds for computational efficiency. Also one can develop similar MILP models to find paths that minimizes the weighted sum of the revisit periods of all the tasks.
\bibliographystyle{IEEEtran}
\bibliography{pisracc.bib}

\end{document}